\documentclass[11pt,twoside,letter]{article}

\usepackage{amsmath}
\usepackage{amsfonts}
\usepackage{amssymb}
\usepackage{fancyhdr}
\usepackage{pst-all}
\input epsf

\newtheorem{theorem}{Theorem}
\newtheorem{corollary}{Corollary}

\newtheorem{lemma}{Lemma}
\newtheorem{proposition}{Proposition}
\newtheorem{remark}{Remark}
\newenvironment{proof}[1][Proof]{\noindent\textbf{#1.} }{\ \rule{0.5em}{0.5em}}

\def\F{\mathcal{F}_{q,v}}
\def\R{\mathbb{R}_q}
\def\RR{\mathbb{R}_q^+}
\def\C{\mathbb{C}}
\def\r{\mathbb{R}}
\def\Z{\mathbb{Z}}
\def\N{\mathbb{N}}
\def\I{\infty}
\def\L{\mathcal{L}_{q,1,v}}
\def\l2{\mathcal{L}_{q,2,v}}
\def\lp{\mathcal{L}_{q,p,v}}
\def\la{\mathcal{L}_{q,a}^v}
\def\La{\mathcal{L}_{q,v,a}}
\def\a{\aligned}
\def\ea{\endaligned}
\def\P{PW_{q,a}^v}
\def\Pb{PW_{q,b}^v}

\begin{document}

\title{\bf\vspace{-39pt} Prolate Spheroidal Wave Functions In $q$-Fourier Analysis}

\author{Lazhar Dhaouadi\\
\small Mathematics Department\\
Institut Pr\'eparatoire aux Etudes d'Ing\'enieur de Bizerte\\
Route Menzel Abderrahmene Bizerte\\
\small Zarzouna, 7021, Tunisia\\
\small lazhardhaouadi@yahoo.fr}

\date{}

\maketitle \thispagestyle{fancy}

\begin{abstract}
In this paper we introduce a new version of the Prolate spheroidal
wave function using standard methods of $q$-calculus and we
formulate some of its properties. As application we give a
$q$-sampling theorem which extrapolates functions defined on $q^n$
and $0<q<1$.
\vspace{5mm} 
\vspace{3mm}
\end{abstract}

\section {Introduction}

The prolate spheroidal wave functions, which are a special case of
the spheroidal wave functions, possess a very surprising and unique
property \cite{W}. They are an orthogonal basis of both $L^2(-1,1)$
and the Paley-Wiener space of bandlimited functions. They also
satisfy a discrete orthogonality relation. No other system of
classical orthogonal functions is known to possess this strange
property. We prove that there are new systems possessing this
property in $q$-Fourier analysis. In the following we discuss some
properties of the $q$-Prolate spheroidal wave function using news
developments and technics in q-Fourier analysis. In particular we
prove that these functions forms an orthogonal basis of the
$q$-Paley-Wiener space $PW_{q,a}^v$. Finally and as application we
give a constructive $q$-sampling formula having as sampling points
$q^n$ where $n\in\Z$. In the end, we cit the reference \cite{A},
where the reproducing kernel for the $q$-Paley-Wiener space was
already discussed, and the explicit formula for the kernel was
given, similar to the formula in Remark \ref{r3}. However, the paper
\cite{A} proceeds with a $q$-sampling theorem which extrapolates
functions defined on the zeros of the $q$-Bessel function. These
zeros are given in the following form
$$
\left\{q^{-n+\epsilon_n}\right\}_{n\in\N},
$$
where $0<\epsilon_n<1$, but it is not explicitly  evaluated.

\section {Preliminary}
Throughout this paper we consider $0<q<1$ and we adopt the standard
conventional notations of \cite{G}. We put
$$
\R=\{\pm q^n,\quad n\in\Z\} ,\quad\RR=\{q^n,\quad n\in\Z\},
$$
and if $a=q^n,\quad n\in\Z$ put
$$
[0,a]_q=\{q^s,\quad s\in\Z,~~s\geq n\}.
$$
For complex $z$, let
$$
(z;q)_0=1,\quad (z;q)_n=\prod_{i=0}^{n-1}(1-zq^{i}), \quad
n=1...\infty.
$$
Jackson's $q$-integral in the interval $[0,a]$ and in the interval
$[0,\I[$ are defined, respectively, by(see \cite{J})
$$
\aligned
\int_0^af(x)d_qx&=(1-q)a\sum_{n=0}^\I q^nf(aq^n),\\
\int_0^\infty f(x)d_qx&=(1-q)\sum_{n=-\I}^\I q^nf(q^n).
\endaligned
$$
For $v>-1$, let $\lp$ be the space of even functions $f$ defined on
$\R$ such that
$$
 \|f\|_{q,p,v}=\left[\int_0^{\infty}|f(x)|^px^{2v+1}d_qx\right]^{1/p}<\infty.
$$
The set $\l2$ is an Hilbert space with the inner product
$$
\langle f,g\rangle=\int_0^\I f(t)g(t)t^{2v+1}d_qt.
$$
We consider $\La$ the space of function defined on $[0,a]_q$ which
satisfies
$$
\int_0^a|f(x)|^2x^{2v+1}d_qx<\I,
$$
and $\la$ the subspace of $\l2$ given by the natural embedding of
$\La$ in $\l2$.

\bigskip

The normalized Hahn-Exton $q$-Bessel function of order $v>-1$ (see
\cite{S}) is defined by
$$
j_v(z,q)=\sum_{n=0}^{\infty }(-1)^{n}\frac{q^{\frac{n(n+1)}{2}}}{%
(q,q)_{n}(q^{v +1},q)_{n}}z^{2n}.
$$
It is an entire analytic function in $z$.

\begin{proposition}\label{p1}
For  $\Re(v)>-1,a>0$ and $y,z\in\C\backslash\{0\}$ we have
$$
\aligned
&\int_0^aj_v(yt,q^2)j_v(zt,q^2)t^{2v+1}d_qt\\
&=\frac{1-q}{1-q^{2v+2}}a^{2v+2}\quad
\frac{y^2j_{v+1}(ay,q^2)j_v(aq^{-1}z,q^2)-z^2j_{v+1}(az,q^2)j_v(aq^{-1}y,q^2)}{y^2-z^2}.
\endaligned
$$
\end{proposition}

\begin{proof}
See \cite{Ko} (Proposition 1.3)
\end{proof}

\bigskip
The following results in this section were proved in  \cite{D}.

\begin{proposition}\label{p2}
$$
|j_v(q^n,q^2)|\leq\frac{(-q^2;q^2)_\I(-q^{2v+2};q^2)_\I}{(q^{2v+2};q^2)_\I}\left\{
\begin{array}{c}
  1\quad\quad\quad\quad\quad\text{if}\quad n\geq 0 \\
  q^{n^2+(2v+1)n}\quad\text{if}\quad n<0
\end{array}
\right..
$$
\end{proposition}

\bigskip
The $q$-Bessel Fourier transform $\F$ introduced in
\cite{D},\cite{K} as follow
$$
\F f(x)=c_{q,v }\int_{0}^{\infty }f(t)j_{v }(xt,q^{2})t^{2v
+1}d_{q}t,
$$
where
$$
c_{q,v }=\frac{1}{1-q}\frac{(q^{2v +2},q^{2})_{\infty }}{%
(q^{2},q^{2})_{\infty }}.
$$
The $q-$Bessel translation operator is defined as follows:
$$
T^v_{q,x}f(y)=c_{q,v}\int_0^\I\F(f)(t)j_v(xt,q^2)j_v(yt,q^2)t^{2v+1}d_qt,\quad\forall
x,y\in\R,\forall f\in\L,
$$
Recall that $T^v_{q,x}$ is said positive if $T^v_{q,x}f\geq 0$ for
$f\geq 0$. In the following we tack $q\in Q_v$ where
$$
Q_v=\{q\in ]0,1[,\quad T^v_{q,x}\quad \text{is positive for
all}\quad x\in\R\}.
$$
The $q-$convolution product of both functions $f,g\in\L$ is defined
by
$$
f*_qg(x)=c_{q,v}\int_0^\I T^v_{q,x}f(y)g(y)y^{2v+1}d_qy.
$$

\begin{theorem}\label{t1}
The operator $\F$ satisfying

\bigskip

1. For all functions $f\in\l2$,\quad $\F^2f(x)=f(x),\quad\forall
x\in\R.$

2. For all functions  $f,g\in\l2$,\quad $\langle \F f,
g\rangle=\langle  f, \F g\rangle$.

3. For all functions  $f\in\l2$,\quad $\|\F
f\|_{q,v,2}=\|f\|_{q,v,2}$.

4. For all functions  $f,g\in\L$,$$\F (f*_qg)(x)=\F f(x)\times\F
g(x),\quad\forall x\in\R.$$
\end{theorem}
In the end we consider $\P$ the $q$-Paley Wiener space
$$
\P=\left\{f(x)=\int_0^a u(t)j_v(xt,q^2)t^{2v+1}d_qt,\quad u\in\la
\right\},
$$
the set of $q$-bandlimited signal.

\section {Main Results}

We introduce the $q$-analogue of the Prolate Spheroidal Wave
Functions $\psi_i$ as the eigenfunction of the integral operator
$T_a^v$ acting on the Hilbert space $\La$ as follows
$$
T_a^vu(x)=c_{q,v}\int_0^a u(t)j_v(xt,q^2)t^{2v+1}d_qt,
$$
then we have
$$
T_a^v\psi_i=\lambda_i\psi_i.
$$
It's easy to see that the operator $T_a^v$ is symmetric and compact
$$
\int_0^a T_a^vu(t)w(t)t^{2v+1}d_qt=\int_0^a
u(t)T_a^vw(t)t^{2v+1}d_qt,
$$
then the sequence $\left\{\psi_i\right\}_{i\in\N}$ forme an
orthogonal basis of the Hilbert space $\La$ and any eigenvalue
$\lambda_i$ is real.
\begin{proposition}\label{p3}
The sequence of eigenvalue $\{\lambda_i\}_{i\in\N}$ satisfying
$$
\lambda^2_0\geq\lambda^2_1\geq\ldots>0.
$$
\end{proposition}

\begin{proof}
The  operator $T_a^v$ is compact, then the spectrum is a countably
infinite subset of $\r$ ($T_a^v$ is symmetric) which has $0$ as its
only limit point. If we denote by
$$
\Lambda=\{\lambda_0,\lambda_1,\ldots\},
$$
the spectrum of $T_a^v$ then we can write
$$
|\lambda_0|\geq|\lambda_1|\geq\ldots\geq 0.
$$
To finish the proof, if suffice to prove that $0\notin\Lambda$. In
fact if $T_a^v\psi=0$ then $\F\psi$ is an entire function which
vanishes on $[0,a]_a$. By the identity theorem for analytic
functions, $\F\psi=0$ everywhere and thus $\psi=0$.
\end{proof}

\begin{remark}\label{r1}
Consider the operator
$$
k_a^v=T_a^v\circ T_a^v,
$$
then $K_a^v$ is an integral operator acting on the Hilbert space
$\La$ as follows
$$
k_a^vu(x)=\int_0^au(y)k(x,y)y^{2v+1}d_qy,
$$
where
$$
k(x,y)=c_{q,v}^2\int_0^aj_v(xt,q^2)j_v(yt,q^2)t^{2v+1}t^{2v+1}d_qt.
$$
The function $\psi_i$ is an eigenfunction of $k_a^v$
$$
k_a^v\psi_i=\lambda_i^2\psi_i.
$$
\end{remark}

\begin{lemma}\label{l1}
The function $\psi_i$ initially defined on $\R$ can be extended as
an analytic function on $\C$.
\end{lemma}

\begin{proof}The result follows from the relation
$$
\psi_i(z)=\frac{1}{\lambda_i}c_{q,v}\int_0^a
\psi_i(t)j_v(zt,q^2)t^{2v+1}d_qt,
$$
and the fact that  $j_v(.,q^2)$ is an entire function.
\end{proof}

\begin{proposition}\label{p4}
The function $\psi_i$ belonging to the Paley-Wiener space $\P$
\end{proposition}

\begin{proof}Let
$$
\phi_i(x)=\frac{1}{\lambda_i}\psi_i(x)\chi_{[0,a]}(x),
$$
then
$$\a
\F\phi_i(x)&=c_{q,v}\int_0^\I \phi_i(t)j_v(xt,q^2)t^{2v+1}d_qt\\
&=\frac{c_{q,v}}{\lambda_i}\int_0^a \psi_i(t)j_v(xt,q^2)t^{2v+1}d_qt=\psi_i(x),\\
\ea$$ which implies that $\psi_i\in\P$.
\end{proof}

\bigskip
In the following we assume that
$$
\|\psi_i\|^2_{q,2,v}=\langle\psi_i,\psi_i\rangle=1.
$$

\begin{proposition}\label{p5}
The sequence $\left\{\psi_i\right\}_{i\in\N}$ forme an orthonormal
basis of $\P$.
\end{proposition}

\begin{proof}
The $q$-Bessel Fourier transform
$$
\F:\la\rightarrow \P,
$$
define an isomorphism, and the sequence $\{\phi_i\}_{i\in\N}$ form
an orthogonal basis of the Hilbert space $\la$, which lead to the
result.
\end{proof}

\begin{proposition}\label{p6}
Let
$$
k_x:y\mapsto k(x,y),
$$
then
$$
f\in\P\Leftrightarrow f(x)=\langle f,k_x\rangle,\quad\forall x\in\R.
$$
\end{proposition}

\begin{proof}Let
$$
\sigma_a(y)=\F\left(\chi_{[0,a]}\right)(x)=c_{q,v}\int_0^a
j_v(ty,q^2)t^{2v+1}d_qt,
$$
therefore
$$
T^v_{q,x}\sigma_a(y)=c_{q,v}\int_0^a
j_v(tx,q^2)j_v(ty,q^2)t^{2v+1}d_qt=\frac{1}{c_{q,v}}k(x,y),
$$
and then
$$\a
f\in\P&\Leftrightarrow \F f(x)=\F f(x) \chi_{[0,a]}(x)=\F f(x)\F \sigma_a(x)\\
&\Leftrightarrow f(x)=f*_q\sigma_a(x)=c_{q,v}\langle
f,T^v_{q,x}\sigma_a\rangle=\langle f,k_x\rangle. \ea$$ This finish
the proof
\end{proof}

\begin{corollary}\label{c1}
We have
$$
k(x,y)=\sum_{i=0}^\I\psi_i(x)\psi_i(y),\quad\forall x,y\in\R.
$$
\end{corollary}

\begin{proof}In fact $k_x\in\P$. Then
$$
k_x(y)=\sum_{i=0}^\I\langle k_x,\psi_i\rangle\psi_i(y).
$$
On the other hand
$$
\psi_i\in\P\Leftrightarrow \langle \psi_i,k_x\rangle=\psi_i(x),
$$
which prove the result.
\end{proof}

\begin{lemma}\label{l2}
For $i,j\in\N$
$$
\int_0^a\psi_i(x)\psi_j(x)x^{2v+1}d_qx=\lambda_i\lambda_j\delta_{ij}.
$$
\end{lemma}

\begin{proof}In fact
$$
\langle \phi_i,\phi_j\rangle=\langle\F\phi_i,\F\phi_j\rangle=\langle
\psi_i,\psi_j\rangle,
$$
and
$$
\langle
\phi_i,\phi_j\rangle=\frac{1}{\lambda_i\lambda_j}\int_0^a\psi_i(x)\psi_j(x)x^{2v+1}d_qx.
$$
On the other hand, if $i\neq j$ then
$$
\langle
\phi_i,\phi_j\rangle=\int_0^a\phi_i(t)\phi_j(t)t^{2v+1}d_qt=0.
$$
Moreover, $\|\phi_i\|_{q,2,v}=\|\psi_i\|_{q,2,v}=1$ which prove that
$\langle\phi_i,\phi_j\rangle=\delta_{ij}$. This leads to the result.
\end{proof}

\bigskip
In order to be more precise about what it means for the energy of a
$q$-bandlimited single  $f\in\P$ to be mainly concentrated on the
interval $[0,a]_q$, we consider the concentration index:
$$
\theta_a^vf=\frac{\int_0^af(x)^2x^{2v+1}d_qx}{\|f\|^2_{q,v,2}},
$$
whose values range from $0$ to $1$.

\begin{proposition}\label{p7}
The maximum value of $\theta_a^vf $ is attained for $f=\psi_0$ and
$$
\theta_a^vf=\frac{\sum_{i=0}^n \lambda_i^2\langle
f,\psi_i\rangle^2}{\sum_{i=0}^n \langle f,\psi_i\rangle^2}\geq
\lambda_n^2,\quad\text{if}\quad
f\in\text{span}\{\psi_0,\ldots,\psi_n\},
$$

$$
\theta_a^vf=\frac{\sum_{i=n+1}^\I \lambda_i^2\langle
f,\psi_i\rangle^2}{\sum_{i=n+1}^\I \langle f,\psi_i\rangle^2}\leq
\lambda_{n+1}^2,\quad\text{if}\quad
f\in\text{span}\{\psi_0,\ldots,\psi_n\}^\perp.
$$
\end{proposition}

\begin{proof}With the Parseval equality
$$
\int_0^a f(x)^2x^{2v+1}d_qx=\sum_{i=0}^\I \langle f,\phi_i\rangle^2,
$$
and the fact that
$$\a
\sum_{i=0}^\I \langle f,\phi_i\rangle^2
&=\sum_{i=0}^\I \langle\F f,\psi_i\rangle^2\\
&=\sum_{i=0}^\I \lambda_i^2\langle\F f,\phi_i\rangle^2
=\sum_{i=0}^\I \lambda_i^2\langle f,\psi_i\rangle^2, \ea$$

$$
\|f\|_{q,v,2}^2=\sum_{i=0}^\I\langle f, \psi_i\rangle^2,
$$
We get
$$
\theta_a^vf=\frac{\sum_{i=0}^\I \lambda_i^2\langle
f,\psi_i\rangle^2}{\sum_{i=0}^\I \langle
f,\psi_i\rangle^2}\leq\lambda_0^2=\theta_a^v\psi_0,
$$
which leads to the result.
\end{proof}

\begin{remark}\label{r2}
If $b>a$ then
$$
\P\subset\Pb,
$$
Now let $\{\mu_n\}_{n\in\Z}$ the sequence of eigenvalues of the
operator $T_b^v$ then we have
$$
\lambda_0^2=\theta_a^v\psi_0\leq\theta_b^v\psi_0\leq\mu_0^2.
$$
\end{remark}

\begin{proposition}\label{p8}
The $q$-Paley-Wiener space $\P$ is a closed subspace of $\l2$.
\end{proposition}

\begin{proof}First we show that $\P$ is a subspace
of $\l2$. In fact let $$f\in\P$$ then there exist $u\in\la$ such
that
$$
f(x)=c_{q,v}\int_0^a u(t)j_v(xt,q^2)t^{2v+1}d_qt=\F (u)(x).
$$
As $\la\subset\l2$ and from the Theorem \ref{t1} we show that $\F
(u)\in\l2$ which implies
$$
\P\subset\l2.
$$
Now, given $f\in\l2$ and let $\{f_n\}_{n\in\N}$ be a sequence of
element of $\P$ which converge to $f$ in $L^2$-norm. For $n\in\N$,
there exist $u_n\in\la$ such that
$$
f_n(x)=c_{q,v}\int_0^a u_n(t)j_v(xt,q^2)t^{2v+1}d_qt.
$$
Moreover
$$
\lim_{n\rightarrow\I}\|f_n-f\|_{q,2,v}=0,
$$
this give
$$
\lim_{n\rightarrow\I}\|\F f_n-\F f\|_{q,2,v}=0,
$$
and then
$$
\int_0^a|\F f_n(x)-\F f(x)|^2x^{2v+1}d_qx+\int_a^\I|\F
f(x)|^2x^{2v+1}d_qx\rightarrow 0,
$$
which implies $\F f(x)=0$ if $x\in\R$ and $x>a$ and then $f\in\P$.
\end{proof}

\begin{theorem}\label{t2}
For any function $f\in\P$ we have
\begin{equation}\label{e1}
f(z)=(1-q)\sum_{k\in\Z}q^{2k(v+1)}f(q^k)k_z(q^k),\quad\forall
z\in\C.
\end{equation}
\end{theorem}

\begin{proof}In fact $f$ is an analytic function, and from
Proposition \ref{p6}
$$
f(x)=\langle f,k_x\rangle,\quad\forall x\in\R.
$$
We have
$$\aligned
\langle f,k_x\rangle&=\langle \F f,\F k_x\rangle=c_{q,v}\langle \F
f,j_v(x.,q^2)\chi_{[0,a]}\rangle\\
&=c_{q,v}\int_0^a\F f(t)j_v(xt,q^2)t^{2v+1}d_qt.
\endaligned$$
which prove that
$$
z\mapsto \langle f,k_z\rangle,
$$
is an analytic function. On the other hand
$$
\langle f,k_z\rangle=(1-q)\sum_{k\in\Z}q^{2k(v+1)}f(q^k)k_z(q^k),
$$
and
$$
\langle f,k_{q^k}\rangle=f(q^k),\quad\forall k\in\Z.
$$
As $\{0\}$ is an accumulation point of the following set
$$
\{q^k,\quad k\in\Z\},
$$
we conclude that $\langle f,k_z\rangle=f(z),\quad\forall z\in\C.$
\end{proof}

\begin{remark}\label{r3}
In many fields , telecommunication in particular, the
Whittaker-Shannon-Kotel'nikov sampling theorem plays a central role.
It is know that sampling is the process of converting a signal
(e.g., a function of continuous time or space) into a numeric
sequence (a function of discrete time or space). Namely this theorem
say that every function in the cosine Paley-Wiener space:
$$
PW^{-\frac{1}{2}}_a=\left\{f(x)=\sqrt{\frac{2}{\pi}}\int_0^{a}u(t)\cos(xt)dt,\quad
u\in  L^2[0,a]\right\},
$$
can be written as
$$f(x)=\sqrt{\frac{2}{\pi}}\sum_{n\in\Z}f\left(\frac{\pi}{a}n\right)\frac{\sin(ax-\pi
n)}{ax-\pi n}.
$$
Then the above theorem can be viewed as a sampling formula where the
sampling points are $q^n$ independent of $a$. By the use of
Proposition 1 we get
$$
k_z(q^n)=\frac{(1-q)c_{q,v}^2}{1-q^{2v+2}}a^{2v+2}\times
\frac{q^{2n}j_{v+1}(aq^n,q^2)j_v(aq^{-1}z,q^2)-z^2j_{v+1}(az,q^2)j_v(aq^{-1+n},q^2)}{q^{2n}-z^2}.
$$
\end{remark}

\begin{proposition}\label{p9}
Given a function $f\in\l2$ and let
$$
f_a(x)=\langle f,k_x\rangle,
$$
then
$$
f_a\in\P,
$$
and for all $\delta>0$ we have
$$
\lim_{a\rightarrow\I}\sup_{x>\delta,x\in\R}|f(x)-f_a(x)|=0.
$$
\end{proposition}

\begin{proof}First
$$
|f_a(x)|\leq\|f\|_{q,v,2}\|k_x\|_{q,v,2}<\I.
$$
Now we can write
$$\aligned
f_a(x)&=\langle f,k_x\rangle=\langle \F f,\F
k_x\rangle=c_{q,v}\langle \F
f,j_v(x.,q^2)\chi_{[0,a]}\rangle\\
&=c_{q,v}\int_0^a\F f(t)j_v(xt,q^2)t^{2v+1}d_qt.
\endaligned$$
which prove that $f_a\in\P$. On the other hand
$$
f(x)= c_{q,v}\langle \F f,j_v(x.,q^2)\rangle,
$$
and therefore
$$\a
|f(x)-f_a(x)|^2&=c_{q,v}^2\left|\int_a^\I \F
f(t)j_v(xt,q^2)t^{2v+1}d_qt\right|^2\\
&\leq c_{q,v}^2\left(\int_a^\I |\F f(t)||j_v(xt,q^2)|t^{2v+1}d_qt\right)^2\\
&\leq c_{q,v}^2\int_a^\I |\F f(t)|^2t^{2v+1}d_qt\int_a^\I
|j_v(xt,q^2)|^2t^{2v+1}d_qt\\
&\leq \frac{c_{q,v}^2}{x^{2v+2}}\int_a^\I |\F
f(t)|^2t^{2v+1}d_qt\int_{ax}^\I |j_v(t,q^2)|^2t^{2v+1}d_qt\\
&\leq \frac{c_{q,v}^2\|j_v(.,q^2)\|^2_{q,v,2}}{x^{2v+2}}\int_a^\I
|\F f(t)|^2t^{2v+1}d_qt. \ea$$ Using the fact that
$$
\int_0^\I |\F f(t)|^2t^{2v+1}d_qt=\|\F f\|_{q,v,2}^2=\|
f\|_{q,v,2}^2<\I,
$$
we finish the proof.
\end{proof}

\section{Application}
In this section we tack $v=-1/2$ and $q=0.5$ and we put
$$
f(x)=\frac{1}{1+x^2},
$$
an even function belong to the space $\l2$. Using the sampling
formula (\ref{e1}) for the function $f_a(x)=\langle f,k_x\rangle$
respectively for $a=1,$ $a=1/q$ and $a=1/q^2$ with sampling point
$$q^n,\quad n=-1\ldots 10$$ we obtain

\bigskip
\begin{pspicture}(10,4)
\scalebox{0.5}{

\put(-1,0){\epsfbox{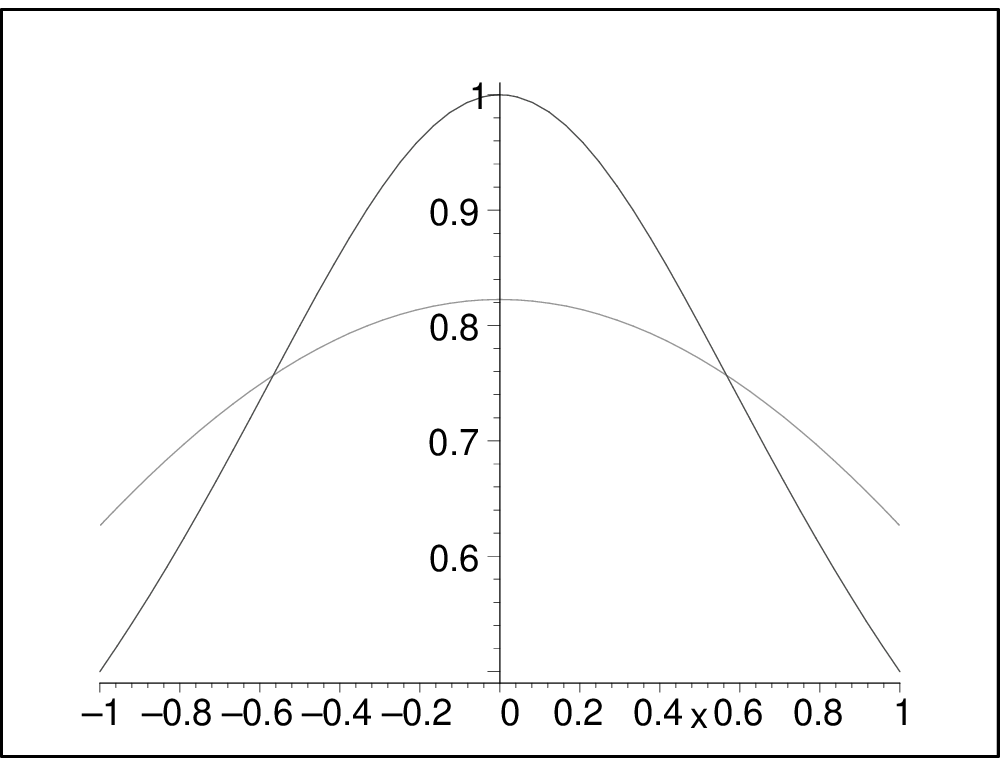}}

\put(9,0){\epsfbox{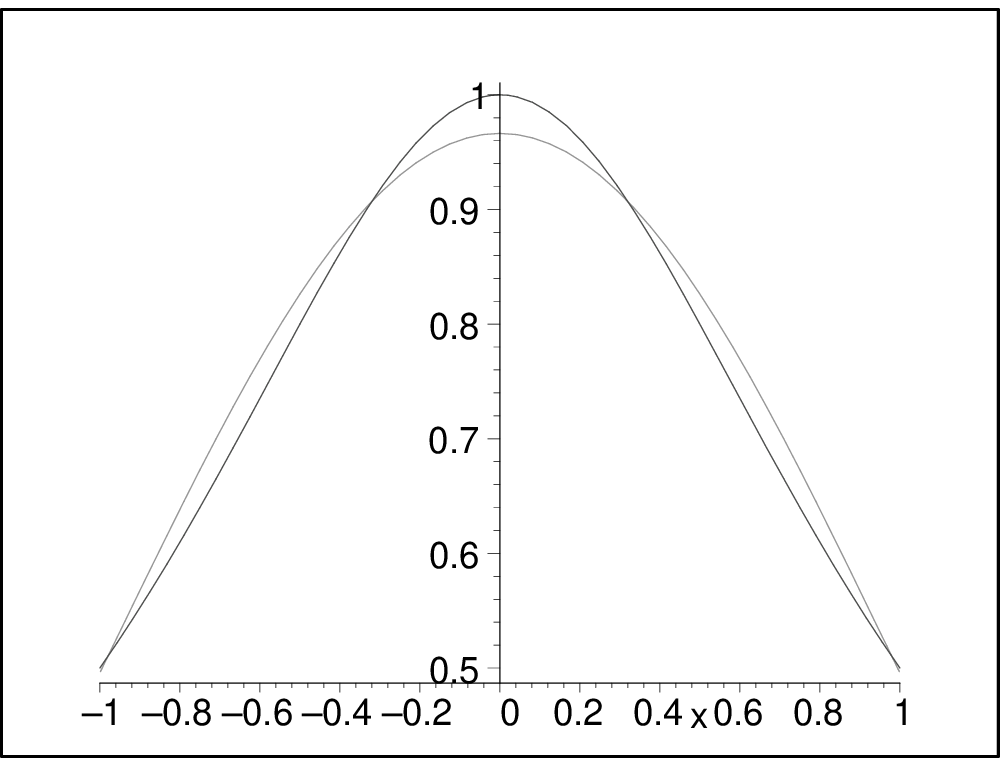}}

\put(19,0){\epsfbox{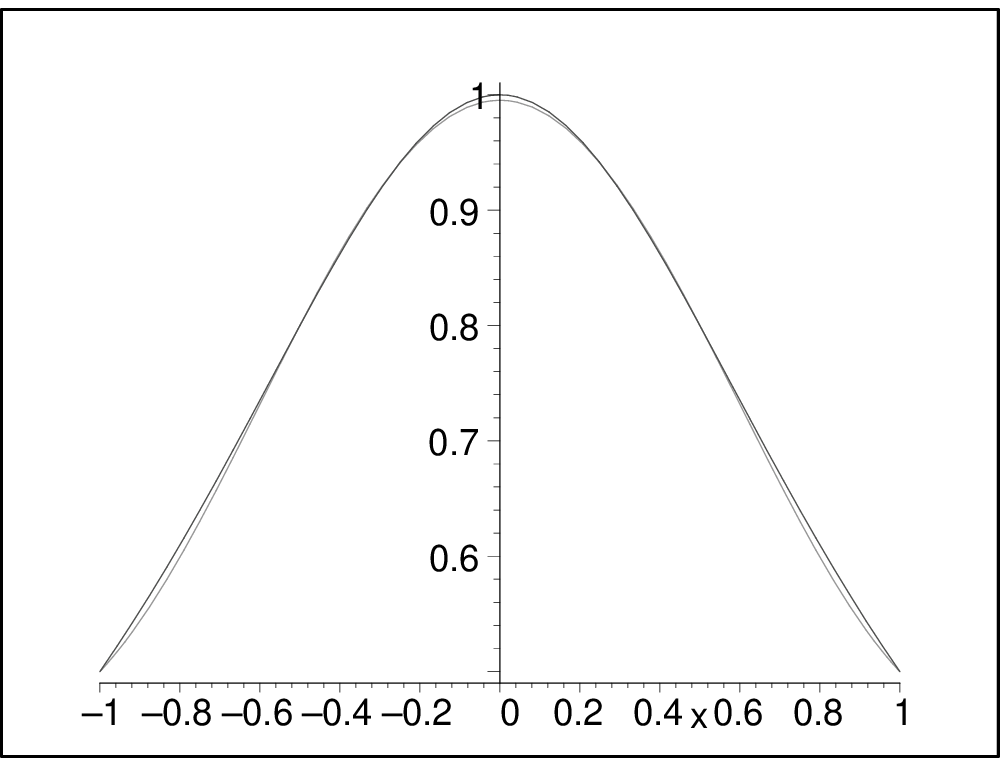}} }
\end{pspicture}

\end{document}